\documentclass[10pt]{amsart}
\usepackage[margin=3.3cm]{geometry}
\usepackage{eepic}
\usepackage{epsfig}
\usepackage{graphicx}
\usepackage{eufrak}
\usepackage{mathrsfs}
\usepackage[english]{babel}
\usepackage{color}
\usepackage{mathabx}

\usepackage{palatino, mathpazo}
\usepackage{amscd}      
\usepackage{amssymb}
\usepackage{amsmath}
\usepackage{dsfont}
\usepackage{xypic}      
\LaTeXdiagrams          
\usepackage[all,v2]{xy}
\xyoption{2cell} \UseAllTwocells \xyoption{frame} \CompileMatrices
\usepackage{latexsym}
\usepackage{epsfig}
\usepackage{amsfonts}
\usepackage{enumerate}

\allowdisplaybreaks[3]

\newtheorem{prop}{Proposition}[section]
\newtheorem{lem}[prop]{Lemma}

\newtheorem{thm}[prop]{Theorem}
\newtheorem{rmk}[prop]{Remark}

            {\nolinebreak $\Box$ \end{trivlist}}

\newcommand{\noprint}[1]{}

\renewcommand{\tilde}{\widetilde}

\newcommand{\Ext}{\mbox{Ext}}

\newcommand{\MM}{{\mathfrak M}}

\newcommand{\zz}{{\mathbb Z}}

\newcommand{\qq}{{\mathbb Q}}
\newcommand{\pp}{{\mathbb P}}
\newcommand{\cc}{{\mathbb C}}

\newcommand{\sC}{{\mathcal C}}
\newcommand{\sE}{{\mathcal E}}

\newcommand{\sL}{{\mathcal L}}

\newcommand{\sS}{{\mathcal S}}

\newcommand{\sO}{{\mathcal O}}

\newcommand{\sM}{{\mathcal M}}

\newcommand{\bfQ}{\mathbf{Q}}
\newcommand{\bbQ}{\mathfrak{Q}}
\newcommand{\tL}{\mathscr{L}}
\newcommand{\tE}{\mathscr{E}}
\newcommand{\scrV}{\mathscr{V}}
\newcommand{\scrF}{\mathscr{F}}

\DeclareMathOperator{\Aut}{Aut}
\DeclareMathOperator{\SP}{Sp}

\DeclareMathOperator{\vir}{vir}

\DeclareMathOperator{\Pic}{Pic}
\DeclareMathOperator{\vd}{vd}
\DeclareMathOperator{\Ch}{Ch}
\DeclareMathOperator{\CR}{CR}

\DeclareMathOperator{\End}{End}

\DeclareMathOperator{\SL}{SL}
\DeclareMathOperator{\PGL}{PGL}

\DeclareMathOperator{\Prym}{Prym}

\DeclareMathOperator{\SU}{SU}

\DeclareMathOperator{\Higgs}{Higgs}
\DeclareMathOperator{\age}{age}
\DeclareMathOperator{\Sp}{Sp}
\DeclareMathOperator{\orb}{orb}
\DeclareMathOperator{\quantum}{quantum}
\DeclareMathOperator{\Td}{Td}
\DeclareMathOperator{\constant}{constant}

\newcommand{\on}{\mathbin{\mid}}
\newcommand{\rk}{\mathop{\rm rk}}

\newcommand{\ev}{\mathop{\rm ev}\nolimits}

\newcommand{\num}{\mathop{\rm num}\nolimits}

\numberwithin{equation}{subsection}

\newcommand {\mat}      [1] {\left(\begin{array}{#1}}
\newcommand {\rix}          {\end{array}\right)}

\title[Quantum cohomology of moduli space of $\PGL_2$-bundles on curves]{The quantum cohomology of moduli space of $\PGL_2$-bundles on curves}

\author{Sagnik Das}
\address{Department of Mathematics\\ University of Kansas\\ 405 Snow Hall 1460 Jayhawk Blvd\\Lawrence KS 66045 USA} 
\email{das.sagnik@ku.edu}

\author{Yunfeng Jiang}
\address{Department of Mathematics\\ University of Kansas\\ 405 Snow Hall 1460 Jayhawk Blvd\\Lawrence KS 66045 USA} 
\email{y.jiang@ku.edu}

\author{Hsian-Hua Tseng}
\address{Department of Mathematics\\ Ohio State University\\ 100 Math Tower, 231 West 18th Ave. \\ Columbus,  OH 43210\\ USA}
\email{hhtseng@math.ohio-state.edu}

\begin{document}
\sloppy \maketitle
\begin{abstract}
We calculate the quantum cohomology of the moduli space of stable $\PGL_2$-bundles over a smooth  curve of genus $g\ge 2$. 
\end{abstract}

\maketitle

\tableofcontents

\section{Introduction}

We work over  $\cc$ throughout of the paper. Let $X$ be a smooth algebraic curve of genus $g\ge 2$. The main goal of this paper is to calculate the quantum cohomology of the moduli space of stable $\PGL_2$-bundles over  $X$.

\subsection{Background}

Let $M_{n,L}(X)$ be the moduli space of semi-stable vector bundles $E$ with rank $n$ and determinant $\det(E)=L\in\Pic(X)$. 
Let $d=\deg(L)$ be the degree of $L$. When $\gcd(n,d)=1$, semi-stability of $E$ is equivalent to stability, and $M_{n,L}(X)$ is a smooth projective variety of dimension $(n^2-1)(g-1)$.  
For simplicity we assume\footnote{For $\deg(L)=d$ odd, $M_{n,L}(X)$ is isomorphic to $M_{n,L'}(X)$ where $\deg(L')=1$, by tensoring.} that $d=\deg(L)=1$. 

The Langlands dual group of $\SL_n(\cc)$ is $\PGL_n(\cc)$. Let  $M^{\PGL_n}_{n,d}(X)$ be the moduli space of stable $\PGL_n$-bundles $E$ with $\deg(E)=d$.  We let $1\le d\le n-1$.  Let $$\Gamma:=\Pic^0(X)[n]$$ be the  subgroup of $n$-torsions of the degree zero Picard group of $X$. Then $$M^{\PGL_n}_{n,d}(X)=[M_{n,L}(X) /\Gamma]$$ is a smooth global quotient Deligne-Mumford stack, where the action of $\Gamma$ is given by tensor product. 

We can generalize the moduli spaces above to the case of Higgs bundles. 
Let $\sM^{\Higgs}_{n,L}(X)$ be the moduli space of stable $\SL_n$-Higgs bundles $(E,\varphi)$ for $\varphi: E\to E\otimes K_{X}$ over an algebraic curve $X$ with rank $n$, determinant $L\in\Pic(X)$. 
This moduli space $\sM^{\Higgs}_{n,L}(X)$ is a smooth quasi-projective variety.  It is a hyperk\"ahler variety which admits three complex structures. 
Similar to the above case of stable bundles on curves, 
the moduli space $\sM^{\Higgs}_{n,d}(\PGL_n)$ of $\PGL_n$-Higgs bundles of rank $n$ and degree $d$ is the smooth quotient Deligne-Mumford stack 
$[\sM^{\Higgs}_{n,L}(X)/\Gamma]$. 

Mirror symmetry in mathematics,  motivated by theoretical physics, is a very important and interesting research  direction in the last several decades.  Its influence 
goes over many research topics including Gromov-Witten theory, Donaldson-Thomas theory, and even recently the Vafa-Witten theory and S-duality, see \cite{TT1}, \cite{Jiang_Adv}, \cite{JK}. 
From \cite{HT},  the moduli spaces $\sM^{\Higgs}_{n,L}(X)$ and $\sM^{\Higgs}_{n,d}(\PGL_n)$ provided the examples 
of  nontrivial mirror pair manifolds.  
Hausel-Thaddeus \cite{HT} conjectured that the stringy Hodge numbers of $\sM^{\Higgs}_{n,L}(X)$ and $\sM^{\Higgs}_{n,d}(\PGL_n)$ are the same. 
The conjecture was first proven recently by 
Groechenig-Wyss-Ziegler \cite{GWZ} using the techniques of $p$-adic integration.  In the paper \cite{MS}, Maulik-Shen gave a new proof for the conjecture  using Ng\^{o}'s support theorem  and endoscopic decomposition in \cite{Ngo_IHES} for the morphism of Hitchin fibrations. 

We are interested in the enumerative geometry for the spaces $\sM^{\Higgs}_{n,L}(X)$ and $\sM^{\Higgs}_{n,d}(\PGL_n)$, see \cite{Nesterov} for a recent progress of the genus one quasi-map invariants and Vafa-Witten invariants.  
The topology and geometry are already very interesting in the case of moduli space of stable bundles on curves.   In this paper we focus on the case of the moduli space of stable bundles in rank $2$.  
The quantum cohomology of the moduli space  $M_{2,L}(X)$ of stable rank $2$ bundles on $X$ was studied in \cite{Donaldson}, and \cite{Munoz}.  We complete the duality and calculate the quantum cohomology of the moduli space  $M^{\PGL_2}_{2,d}(X)$. 
For the moduli space $M^{\PGL_2}_{2,d}(X)=[M_{2,L}(X)/\Gamma]$,   Harder-Narasimhan \cite{Harder-Narasimhan} proved that the cohomology $H^*(M)$ is isomorphic to the cohomology of $H^*(M/\Gamma)$. 
The orbifold cohomology and quantum orbifold cohomology of $M^{\PGL_2}_{2,d}(X)$ are not isomorphic to the cohomology and quantum cohomology of $M=M_{2,L}(X)$ anymore, due to the nontrivial  contribution of twisted sectors of $[M/\Gamma]$. The authors believe 
the quantum orbifold cohomology of $M^{\PGL_2}_{2,d}(X)$ already contains information of the quantum cohomology of the moduli space $\sM^{\Higgs}_{2,L}(X)$ of Higgs bundles. 
In a sequel paper we will consider the quantum orbifold cohomology of the moduli space  $\sM^{\Higgs}_{2,d}(\PGL_2)$ rank $2$ Higgs bundles on $X$ and the possible mirror symmetry properties. 

\subsection{Result} 

The moduli stack $M^{\PGL_2}_{2,d}(X)$ is the Deligne-Mumford stack $[M_{2,L}(X)/\Gamma]$. Here in the rank $2$ case, $\Gamma=\text{Pic}^0(X)[2]\simeq (\mathbb{Z}/2\mathbb{Z})^{2g}$.  Let $\kappa\in \Gamma$ be an element. Then the 
$\kappa$-fixed locus $(M_{2,L}(X))^{\kappa}$ is the Prym variety 
$\Prym(X^\prime/X)$, where $X^\prime\to X$ is an \'etale double cover of $X$.  
The Prym variety $\Prym(X^\prime/X)$ is a smooth abelian variety of dimension $g-1$. 
More details can be found in \S \ref{sec_PGL2}.
The inertia stack $IM^{\PGL_2}_{2,d}(X)$ has the following decomposition 
$$IM^{\PGL_2}_{2,d}(X)=M^{\PGL_2}_{2,d}(X)\sqcup \bigsqcup_{0\neq \kappa\in \Gamma}[(M_{2,L}(X))^{\kappa}/\Gamma].$$
Thus, for each nontrivial  $\kappa\in \Gamma$, we have the stack $[(M_{2,L}(X))^{\kappa}/\Gamma]\cong [\Prym(X^\prime/X)/\Gamma]$.  
We let $\age(\kappa)$  be the age (or the degree shifting number) of the component (twisted sector) $[(M_{2,L}(X))^{\kappa}/\Gamma]$. 
The Chen-Ruan orbifold cohomology of $M^{\PGL_2}_{2,d}(X)$ is additively given by 
\begin{equation}\label{eqn_Chen-Ruan}
H^*_{\CR}(M^{\PGL_2}_{2,d}(X))=H^*(M^{\PGL_2}_{2,d}(X)) \oplus\bigoplus_{0\neq \kappa\in \Gamma}
H^{*-2\age(\kappa)}[(M_{2,L}(X))^{\kappa}/\Gamma].
\end{equation}
The Chen-Ruan orbifold cup product of $M^{\PGL_2}_{2,d}(X)$ was calculated in \cite{BP}.  

The quantum orbifold cohomology of $M^{\PGL_2}_{2,d}(X)\simeq [M_{2,L}(X)/\Gamma]$ is defined by the $3$-pointed genus zero Gromov-Witten invariants of $[M_{2,L}(X)/\Gamma]$.
To simplify notation, let $$M:=M_{2,L}(X)$$ and  $M^{\PGL_2}:=M^{\PGL_2}_{2,d}(X)=[M/\Gamma]$. 
Let $$\overline{M}_{0, 3}([M/\Gamma], A)$$ be the moduli space of stable maps from $3$-marked, genus zero twisted curves to 
$[M/\Gamma]$ of degree $A\in H_2([M/\Gamma])$. For $\alpha_1, \alpha_2, \alpha_3\in H^*_{\CR}(M^{\PGL_2}_{2,d}(X))$, the orbifold Gromov-Witten invariant is defined as
$$
\langle \alpha_1, \alpha_2, \alpha_3\rangle_{0,3,A}^{[M/\Gamma]}:=
\int_{[\overline{M}_{0, 3}([M/\Gamma], A)]^{\vir}}\prod_{i=1}^{3}ev_i^*(\alpha_i)
$$
where 
$$
ev_i: \overline{M}_{0, 3}([M/\Gamma], A)\to I[M/\Gamma]
$$
are the evaluation maps  sending 
$$(f: (\sC, \Sigma_i)\to [M/\Gamma])\mapsto (f(p_i), \sigma_i)\in I[M/\Gamma].$$
Here $(\sC, \Sigma_i)$ is the marked twisted genus zero curve, $\Sigma_i (i=1,2,3)$ are the marked $\mu_{r_i}$-gerbes with underlying geometric points $p_i$ on $\sC$ and 
$\sigma_i$ are the images of the injective morphisms $\sigma_i:  \Sigma_i\to \Aut(f(p_i))$.  
For classes $\alpha_1, \alpha_2\in H^*_{\CR}([M/\Gamma])$, let $\alpha_1\cdot_{\bbQ} \alpha_2$ be the small quantum product,  which  is defined by orbifold Popincar\'e pairing $\langle -,- \rangle$,
\begin{equation}\label{eqn_quantum_product}
\langle\alpha_1\cdot_{\bbQ}\alpha_2, \alpha_3\rangle=\sum_{A}\langle \alpha_1, \alpha_2, \alpha_3\rangle_{0,3,A}^{[M/\Gamma]}\bbQ^{A}
\end{equation}
for any $\alpha_3\in H^*_{\CR}([M/\Gamma])$, where $\bbQ$ is the quantum parameter.

We recall the cohomology $H^*(M)$ for the moduli space $M=M_{2,L}(X)$ of stable bundles. The cohomology ring of the moduli space $M_{n,L}(X)$ for arbitrary rank $n$ was calculated in \cite{Kirwan}, \cite{Earl-Kirwan}.
We recall the ring structure described in \cite{King-Newstead} in the rank $2$ case.  Let $\sE\to X\times M$ be the universal rank $2$ vector bundle. Using Kunneth decomposition, we can write the second Chern class of  the endomorphism of $\sE$ as
\begin{equation}\label{eqn_second_Chern}
c_2(\End_0(\sE))=2[X]\otimes \alpha+4\psi-\beta.
\end{equation}
We explain the notations.  
Let $\{\xi_1, \cdots, \xi_{2g}\}\in H^1(X,\zz)$ be a basis and $\{\xi^*_1, \cdots, \xi^*_{2g}\}$ be the dual basis in $H_1(X,\zz)$.
Then $\xi^*_i\xi^*_{i+g}=[X]\in H_2(X,\zz)$ for $1\le i\le g$. 
We have  $\psi=\sum_{i=1}^{2g}\xi_i\otimes \psi_i$.  The classes $\alpha, \beta, \psi_i$ are given as follows.  
Let 
$$\mu: H_*(X)\to H^{4-*}(M)$$
be the map given by 
$\mu(a)=-\frac{1}{4} p_1(\mathfrak{g}_{\sE})/a$. Here $\mathfrak{g}_{\sE}\to X\times M$ is the associated universal $SO(3)$-bundle and 
$p_1(\mathfrak{g}_{\sE})\in H^4(X\times M)$ is the first Pontrjagin class. 
Then we have the classes 
$$
\begin{cases}
\alpha=2\mu([X])\in H^2(M), \\
\psi_i=\mu(\xi_i^*)\in H^3(M), & 1\le i\le 2g,\\
\beta=-4\mu(x)\in H^4(M),& x\in H_0(X) \text{~the point class.}
\end{cases}
$$
From \cite{King-Newstead}, these tautological classes $\alpha, \beta, \psi$ generate the cohomology ring 
$H^*(M)$ with relations  defined by recursions. 
Let $\gamma:=-2\sum_{i=1}^{g}\psi_i \psi_{i+g}$. 
Set $q_0^1=1, q_0^2=0, q_0^3=0$, for $r\ge 1$, define 
\begin{equation}\label{eqn_cohomology_relations}
\begin{cases}
q^1_{r+1}=\alpha q_r^1+ r^2 q_r^2, \\
q_{r+1}^2=\beta q_r^1+\frac{2r}{r+1}q_r^3,\\
q_{r+1}^3=\gamma q_r^1.
\end{cases}
\end{equation}
Let $I_g=(q_g^1, q_g^2, q_g^3)\subset \qq[\alpha,\beta,\gamma]$ for all $g\ge 1$. Then we have 
$$H^*(M)=\bigoplus_{k=0}^{g-1}\wedge_0^k H^3(M)\otimes \qq[\alpha, \beta, \gamma]/I_{g-k}.$$

Let $\Sp(2g,\zz)$ be the symplectic group which acts on $H^*(M)$ (actually acting on $\psi_i$). The invariant part of the cohomology ring, $H^*_I(M)$, is generated by $\alpha, \beta, \gamma$ and has the representation
$$H^*_I(M)= \qq[\alpha, \beta, \gamma]/I_{g}.$$

Mu\~noz \cite{Munoz} calculated the quantum cohomology ring of $M=M_{2,L}(X)$.  Let $\bbQ$ be the quantum parameter. First we have the following deformation of the relations in (\ref{eqn_cohomology_relations}).
Set $Q_0^1=1, Q_0^2=0, Q_0^3=0$, for $r\ge 1$, define 
\begin{equation}\label{eqn_Quan_cohomology_relationsAA}
\begin{cases}
Q^1_{2}=\alpha^2+ \beta+(-1)^g8\cdot \bbQ, \\
Q_{2}^2=\alpha(\beta+(-1)^{g}8\bbQ)+\gamma,\\
Q_{2}^3=\alpha\gamma+\alpha^2\cdot\bbQ.
\end{cases}
\end{equation}
and for $r\ge 3$, 
\begin{equation}\label{eqn_Quan_cohomology_relations}
\begin{cases}
Q^1_{r+1}=\alpha Q_r^1+ r^2 Q_r^2, \\
Q_{r+1}^2=(\beta+(-1)^{r+g-1}8) Q_r^1+\frac{2r}{r+1}Q_r^3,\\
Q_{r+1}^3=\gamma Q_r^1.
\end{cases}
\end{equation}
Then we have 
\begin{thm}\label{thm_quan_coh_M}(\cite[Theorem 20]{Munoz})
Assume $g(X)\ge 3$. The quantum cohomology of $M=M_{2,L}(X)$ has the following representation 
$$QH^*(M)=\bigoplus_{k=0}^{g-1}\wedge_0^k H^3(M)\otimes \qq[\alpha, \beta, \gamma,\bbQ]/J_{g-k},$$
where $J_r=(Q_r^1, Q_r^2, Q_r^3)$ for $r\ge 1$. 
\end{thm}

Let us consider the moduli stack 
$M^{\PGL_2}=[M/\Gamma]$.
For any nontrivial $\kappa\in \Gamma$, the cohomology $H^s(\Prym(X^\prime/X)^{\Gamma})=0$ if $s$ is odd, and when $s\le 2(g-1)$ is even,  $H^s(\Prym(X^\prime/X)^{\Gamma})$ is generated by $r_{\kappa}^s:=\tiny\mat{c}2(g-1)\\s\rix$ generators.
We use $1_{\kappa}$ to denote the generator of $H^0(\Prym(X^\prime/X)^{\Gamma})$.
Consider 
$$R_{\kappa}:=\sum_{\substack{s=0\\
s \text{~even}}}^{2(g-1)}r_\kappa^s=2^{2g-3}.$$ 
We let 
$$1_{\kappa}, 1_{\kappa} h, \cdots, 1_{\kappa}h_{r_{\kappa}^2}, 1_{\kappa}h_{r_{\kappa}^2+1}, \cdots, 1_{\kappa}h_{r_{\kappa}^2+r_{\kappa}^4}, \cdots, 1_{\kappa}h_{R_{\kappa}}$$
be the generators of the cohomology $H^{*}(\Prym(X^\prime/X)^{\Gamma})$.
We let $H^*_{I, \CR}([M/\Gamma])$ be the $\SP(2g,\zz)$-invariant Chen-Ruan cohomology. Then we have 
\begin{equation}\label{eqn_Chen-Ruan_I}
H^*_{I, \CR}([M/\Gamma])\cong \left(\qq[\alpha,\beta,\gamma]\oplus\bigoplus_{0\neq {\kappa}\in\Gamma}\qq[1_{\kappa}, 1_{\kappa} h, \cdots, 1_{\kappa}h_{R_{\kappa}}]\right)/(I_g, I_{\orb})
\end{equation}
where $I_{\orb}$ is the relations coming from the Chen-Ruan cup products, see \S \ref{subsec_Chen-Ruan}.
 
 We introduce the new recursive relations:
Set $\bfQ_0^1=1, \bfQ_0^2=0, \bfQ_0^3=0$; $\bfQ_1^1=\alpha, \bfQ_1^2=\beta+(-1)^g8\bbQ, \bfQ_1^3=\gamma$, 
\begin{equation}\label{eqn_Quan_cohomology_relations-2-g>4}
\begin{cases}
\bfQ^1_{2}=\alpha^2+ \beta+(-1)^g8\cdot \bbQ, \\
\bfQ_{2}^2=\alpha(\beta+(-1)^{g}8\bbQ)+\gamma,\\
\bfQ_{2}^3=\alpha\gamma+\alpha^2\cdot\bbQ.
\end{cases}
\end{equation}
and for $r\ge 3$, define 
\begin{equation}\label{eqn_Quan_cohomology_relations}
\begin{cases}
\bfQ^1_{r+1}=\alpha \bfQ_r^1+ r^2 \bfQ_r^2, \\
\bfQ_{r+1}^2=(\beta+(-1)^{r+g-1} \bfQ_r^1+\frac{2r}{r+1}\bfQ_r^3,\\
\bfQ_{r+1}^3=\gamma \bfQ_r^1.
\end{cases}
\end{equation}

The quantum products of $\alpha$ with the cohomology classes of the twisted sectors in $I[M/\Gamma]$ are given as follows. For any $0\neq {\kappa}\in \Gamma$ so that $[M^{\kappa}/\Gamma]=[\Prym(X^\prime/X)/\Gamma]$, and  for an even integer $s$ between $[0, 2(g-1)]$, $H^s(\Prym(X^\prime/X)^{\Gamma})$ is generated by $r_{\kappa}^s:=\tiny\mat{c}2(g-1)\\s\rix$ generators.
We use the notation $h^s$ to represent any of the $r_\kappa^s$ generators in $H^s(\Prym(X^\prime/X)^{\Gamma})$. Therefore, we count $h^2, h^4, \cdots, h^{2(g-2)}, h^{2(g-1)}$ for the representatives in the cohomology 
$H^2(\Prym(X^\prime/X)^{\Gamma}), H^4(\Prym(X^\prime/X)^{\Gamma}), \cdots, H^{2(g-1)}(\Prym(X^\prime/X)^{\Gamma})$ respectively. 

Then we have 
\begin{equation}\label{eqn_Quan_cohomology_product_twisted-sector}
I_{\quantum}:=
\begin{cases}
\alpha\cdot_{\bbQ}1_{\kappa}=1_{\kappa}\cdot \alpha+1_{\kappa}\cdot \bbQ, \\
\alpha\cdot_{\bbQ}1_{\kappa}h^2=1_{\kappa}\cdot \alpha\cdot h^2+1_{\kappa}h^2\cdot \bbQ, \\
\vdots \quad \vdots\quad \vdots\quad\quad \quad \vdots\quad \vdots\quad \vdots\quad\quad \quad \vdots\quad \vdots\quad \vdots\\
\alpha\cdot_{\bbQ}1_{\kappa}h^{s}=1_{\kappa}\cdot \alpha\cdot h^{s}+
1_{\kappa}\cdot h^{s}\cdot\bbQ\\
\vdots \quad \vdots\quad \vdots\quad\quad \quad \vdots\quad \vdots\quad \vdots\quad\quad \quad \vdots\quad \vdots\quad \vdots\\
\alpha\cdot_{\bbQ}1_{\kappa}h^{2(g-2)}=1_{\kappa}\cdot \alpha\cdot h^{2(g-2)}+
1_{\kappa}\cdot h^{2(g-2)}\cdot\bbQ\\
\alpha\cdot_{\bbQ}1_{\kappa}h^{2(g-1)}=
1_{\kappa}\cdot h^{2(g-1)}\cdot\bbQ.
\end{cases}
\end{equation}
Note that the above quantum products preserve degrees.  We recall that\footnote{Here $s$ is even.} 
$$\deg(\alpha)=1; \quad \deg(1_{\kappa})=(g-1); \quad \deg(1_{\kappa}h_s)=(g-1)+\frac{1}{2}s.$$
The quantum parameter $\bbQ$ has algebraic degree $2$ (the degree of the canonical line bundle $K_{[M/
\Gamma]}$).
When there is nontrivial twisted curves as in the above quantum product, there is a factor $\frac{1}{2}$ coming from the local monodromy group $\mu_2$ and we find that $\bbQ$ has algebraic degree $1$.

Let $I^{\bbQ}_{\orb}=(\bfQ_g^1, \bfQ_g^2, \bfQ_g^3, I_{\quantum})$ be the relations defined in (\ref{eqn_Quan_cohomology_relations})
and (\ref{eqn_Quan_cohomology_product_twisted-sector}). 
The following is our  main result:
\begin{thm}\label{thm_quantum_M-g>3}
Assume the genus of $X$ is $g\ge 2$. 
We have the representation of the quantum orbifold cohomology 
$$QH^*_{I, \CR}([M/\Gamma])\cong \left(\qq[\alpha,\beta,\gamma,\bbQ]\oplus\bigoplus_{0\neq \kappa\in\Gamma}\qq[1_{\kappa}, 1_{\kappa} h, \cdots, 1_{\kappa}h_{R_{\kappa}},\bbQ]\right)/(I^{\bbQ}_{\orb}).$$
\end{thm}

Theorem \ref{thm_quantum_M-g>3} is proven in \S \ref{sec_GW} by relating Gromov-Witten invariants of $[M/\Gamma]$ to Donaldson invariants of $\sC\times X$, where $\sC$ is either $\pp^1$ or a stacky $\pp^1$.
Finally we calculate in detail for the case $g=2$ in \S \ref{sec_234}.

\subsection{Outline} 
In \S \ref{sec_PGL2} we review the moduli space of $\PGL_n$-bundles and the  endoscopic decomposition method for the fixed point locus of the torsion subgroup of the Jacobian. In \S \ref{sec_GW} we study genus zero Gromov-Witten invariants of the moduli space of $\PGL_2$-bundles and calculate the quantum orbifold products of $M^{\PGL_2}$. Finally in  \S \ref{sec_234} we calculate explicitly the example when $g=2$. 

\subsection{Convention}
We work over $\cc$ throughout of the paper.  Let $M$ be a smooth variety or a smooth Deligne-Mumford stack,  we always use $\qq$ coefficient for the homology and cohomology of $M$. We denote it by $H^*(M)=H^*(M,\qq)$.

\subsection*{Acknowledgments}

We thank Promit Kundu, Dragos Oprea, and Junliang Shen for valuable discussions.    
This work is partially supported by Simons Foundation Collaboration Grants and KU Research GO Grant.


\section{Moduli space of $\PGL_2$-bundles}\label{sec_PGL2}

In this section we fix a smooth curve $X$ of genus $g=g(X)\ge 2$. We review the moduli stack $M^{\PGL_2}_{d}(X)$ of $\PGL_2$-bundles over $X$, and calculate its inertia stack. 

\subsection{Moduli space of $\PGL_n$-bundles}\label{subsec_PGL}

Let $n\ge 2$ and $d$ be positive integers such that $\gcd(n,d)=1$.  The moduli space $M:=M_{n,L}(X)$ of stable rank $n$ bundles  over $X$ with fixed determinant $L\in\Pic^d(X)$ is a smooth quasi-projective variety of dimension $(n^2-1)(g-1)$.  

Let us recall the determinant line bundles (or the theta line bundles) on $M$.  Let $$\sE\to X\times M$$ be the universal vector bundle.
Let $\pi_1: \sE\to X\times M\to X$ and  $\pi_2: \sE\to X\times M\to M$  be projections.   For any vector bundle $F$ on $X$ define a {\em determinant line bundle} 
$$L_{F}:=\left(\det(R\pi_{2*}(\sE\otimes \pi_1^*F))\right)^{\vee}\in\Pic(M)$$
It is known that if $F\in K_0(X)$ is a $K$-theory class such that $\chi(E\otimes [F])=0$ for some $E\in M$, then the line bundle $L_F$ does not depend on the choice of $\sE$. 
Let $K_0(X)^{\perp}$ be a subgroup of $K_0(X)$ defined by 
$$K_0(X)^{\perp}:=\{[F]\in K_0(X)\on \chi(F\otimes E)=0 \text{~for some~}E\in M\}.$$
Also the $K$-group 
$K_0(X)\cong \zz^2=\{(\rk, \deg)\}$.  Thus, the determinant line bundle construction 
$$K_0(X)^{\perp}\to \Pic(M)$$
factors through the numerical $K$-group $K_0(X)_{\num}^{\perp}$. By Riemann-Roch, the pair $(n^\prime, d^\prime)\in K_0(X)$
is perpendicular to the data $(n,d)$ if and only if
$$nd^\prime+n^\prime d+n n^\prime(1-g)=0.$$
We assume that $\gcd(n,d)=1$, hence the solutions are of the form 
$$(n^\prime, d^\prime)=m\cdot (n, n(g-1)-d), \quad m\in \zz.$$
Thus, there is an isomorphism \cite{DN}
\begin{equation}\label{eqn_phi_picard}
\phi: \zz\to \Pic(M).
\end{equation}
The theta line bundle $\Theta$ is a line bundle corresponding to $1$ in the homomorphism (\ref{eqn_phi_picard}). 
The moduli space $M$ is a Fano variety of index two. Let $K_M$ be the canonical line bundle. We have 
$K_M\cong \Theta^{-2}$. 

Fix positive integers $n, d$ with $\gcd(n,d)=1, 1\le d\le n-1$, let $M^{\PGL_n}:=M^{\PGL_n}_{n,d}(X)$ be the moduli space of stable $\PGL_n$-bundles with degree $d$. This moduli space is the quotient $[M_{n,L}(X)/\Gamma]$, where $\Gamma=\Pic^0(X)[n]$ is the subgroup of $n$-torsions in the Picard group $\Pic^0(X)$ of degree zero. 
By \cite[Proposition 3.1]{BLS}, $\Theta^n$ descends to a line bundle $\hat{\Theta}$ in $M^{\PGL_n}$ and the Picard group $\Pic(M^{\PGL_n})\cong \zz \hat{\Theta}$.  The canonical line bundle $K_{M^{\PGL_n}}$ is isomorphic to $\hat{\Theta}^{-2}$.

\subsection{Fixed loci}\label{subsec_fixed_locus}

We are interested in the global quotient Deligne-Mumford stack $M^{\PGL_n}=[M/\Gamma]$. Let $\kappa\in \Gamma$ be an element. The action of $\Gamma$ on $M$ is given by tensor product $\kappa\cdot E=E\otimes \kappa$. The $\kappa$-fixed locus $M^{\kappa}$ is given by the generalized Prym variety studied in \cite{NR75}.  

We use the idea of relative moduli spaces introduced in \cite{MS} for the moduli space of Higgs bundles and the endoscopic decomposition of moduli space of Hitchin fibrations. Let $\kappa\in \Gamma\simeq (\zz/n\zz)^{2g}$ be a nontrivial element of order $m$. The line bundle $\kappa$ determines a $\zz_m$ Galois cover 
$$\pi: X^\prime\to X.$$
Here $X^\prime$ is a smooth curve of genus $g(X^\prime)=mg-m+1$. Let $n=mr$. 
Let $M_{r,d}(X^\prime)$ be the moduli space of stable rank $r$ bundles with degree $d$ over $X^\prime$.
Then we have the morphism 
\begin{equation}\label{eqn_morphism_q}
q: M_{r,d}(X^\prime)\to \Pic^d(X^\prime), \quad E\mapsto \det(E).
\end{equation}
Let 
\begin{equation}\label{eqn_morphism_pi*}
\pi_*:  \Pic^d(X^\prime)\to \Pic^d(X), \quad L\mapsto \det(\pi_{*}L)
\end{equation}
be the pushforward morphism and 
\begin{equation}\label{eqn_morphism_qpi}
q_{\pi}:=\pi_{*}\circ q: M_{r,d}(X^\prime)\to \Pic^d(X)
\end{equation}
be the composition. 
We define the relative moduli space of rank $r$ and degree $d$ associated with $\pi$ as 
$$M_{r,L}(\pi):=q_{\pi}^{-1}(L).$$

We have 
\begin{prop}\label{prop_relative_moduli}
The moduli space $M_{r,L}(\pi)$ is a disjoint union of $m$ nonsingular isomorphic components 
$$M_{r,L}(\pi)=\bigsqcup_{i=1}^{m}M_i.$$
For each $1\le i, j\le m$, there exists an isomorphism $\phi_{ij}: M_i\stackrel{\sim}{\longrightarrow}M_j$ induced by 
tensoring with a line bundle $\kappa_{ij}\in\Gamma$. 
\end{prop}
\begin{proof}
This is exactly the result in \cite[Proposition 1.1]{MS} when the Higgs fields are zero.
\end{proof}

The group $\zz_m$ acts on the moduli space $M_{r,L}(\pi)$.   
\begin{prop}
We have the following morphisms 
\begin{equation}\label{eqn_quotient_fixed}
M_{r,L}(\pi)\stackrel{q_{M}}{\longrightarrow} M^{\kappa}\subset M,
\end{equation}
where $q_{M}$ is the  quotient map such that 
$$M_{r,L}(\pi)/\zz_m\cong M^{\kappa}.$$
\end{prop}
\begin{proof}
This is actually the special case of Higgs field zero in \cite[\S 1.5]{MS}. The space $M_{r,L}(\pi)$  is a disjoint union of $m$ non-singular isomorphic components 
$M_{r,L}(\pi)=\coprod_{j=1}^{m}M_j$, and the cyclic group $G_{\pi}=\zz_m$ acts on them transitively.  Therefore, we get the quotient space in 
the proposition. 
\end{proof}

When $n$ is a prime number, we have the following 
\begin{prop}\label{prop_prime_n}
Suppose the rank $n$ is a prime integer.  Then for each $\kappa\in\Gamma$, $\kappa$ has order $n$ and 
the $\kappa$ fixed locus $M^{\kappa}\subset M$ is actually the Prym variety of dimension $(n-1)g-n+1$. 
\end{prop}
\begin{proof}
If $n$ is a prime number, then $\kappa$ has order $n$ and $\pi: X^\prime\to X$ is a $\zz_n$-cover. 
$X^\prime$ has genus $ng-n+1$.  The moduli space  $M_{r,d}(X^\prime)$ is $\Pic^d(X^\prime)$, and  the fiber  $M_{r,L}(\pi)$ of the morphism $q_{\pi}$ in (\ref{eqn_morphism_qpi}) is a disjoint union of $m$ isomorphic abelian varieties in which each one is isomorphic to the Prym variety $\Prym(X^\prime/X)$ of dimension $(n-1)g-n+1$. 
\end{proof}

\subsection{Chen-Ruan cohomology}\label{subsec_Chen-Ruan}

The Chen-Ruan orbifold cohomology of $M^{\PGL_n}:=M_{d}^{\PGL_n}(X)$ is additively the cohomology of the inertia stack $IM^{\PGL_n}$. From the results in \S \ref{subsec_fixed_locus}, 
\begin{equation}\label{eqn_inertia_stack}
IM^{\PGL_n}=[M/\Gamma]\sqcup\bigsqcup_{0\neq\kappa\in\Gamma}[M^{\kappa}/\Gamma].
\end{equation}
The Chen-Ruan cohomology is 
\begin{equation}\label{eqn_Chen-Ruan_cohomology}
H^*_{\CR}(M^{\PGL_n})=H^*([M/\Gamma])\oplus\bigoplus_{0\neq\kappa\in\Gamma}H^{*-2\age(\kappa)}([M^{\kappa}/\Gamma]),
\end{equation}
where $\age(\kappa)$ is the age (degree shifting number) of the component $[M^\kappa/\Gamma]$.

The Chen-Ruan cup product is defined using the genus zero, degree zero Gromov-Witten invariants of $M^{\PGL_n}$, as follows. Let $\alpha_1\in H^{*-2\age(\kappa_1)}([M^{\kappa_1}/\Gamma]), \alpha_2\in H^{*-2\age(\kappa_2)}([M^{\kappa_2}/\Gamma])$. We let $\kappa_3\in\Gamma$ be such that $\kappa_1\kappa_2\kappa_3=1$ in $\Gamma$, and $\overline{M}_{0,3}^{\kappa_1,\kappa_2,\kappa_3}(M^{\PGL_n},0)$ be the moduli space of  degree zero twisted stable maps from genus zero, $3$-marked, twisted stable curves to $M^{\PGL_n}$, such that the evaluation maps $\ev_i$ have images contained in $[M^{\kappa_i}/\Gamma]$.
The Chen-Ruan cup product is defined as:
\begin{equation}\label{eqn_Chen-Ruan_product}
\langle\alpha_1\cup_{\CR}\alpha_2, \alpha_3\rangle=
\langle\alpha_1, \alpha_2, \alpha_3\rangle_{0,3,0}^{M^{\PGL_n}}=
\int_{[\overline{M}_{0,3}^{\kappa_1,\kappa_2,\kappa_3}(M^{\PGL_n},0)]^{\vir}}\ev_1^*\alpha_1\cdot \ev_2^*\alpha_2\cdot \ev_3^*\alpha_3,
\end{equation}
where $[\overline{M}_{0,3}^{\kappa_1,\kappa_2,\kappa_3}(M^{\PGL_n},0)]^{\vir}$ is the virtual fundamental class of the moduli space $\overline{M}_{0,3}^{\kappa_1,\kappa_2,\kappa_3}(M^{\PGL_n},0)$ and 
$\langle\alpha_1\cup_{\CR}\alpha_2, \alpha_3\rangle$ is the orbifold pairing. 

The virtual fundamental class $[\overline{M}_{0,3}^{\kappa_1,\kappa_2,\kappa_3}(M^{\PGL_n},0)]^{\vir}$ is given by the Euler class of the obstruction bundle \cite{CR}. This obstruction bundle may be described as follows. We write the tangent bundle $T_{M^{\PGL_n}}$ as the decomposition under the group $\langle \kappa_i\rangle$-representation:
$$T_{M^{\PGL_n}}=\bigoplus_{0\le a_i<1}L_{a_i},$$
where $L_{a_i}$ means the subbundle that $\kappa_i$ acts as $e^{2\pi i a_i}$. 
Then 
\begin{equation}\label{eqn_obstruction_bundle}
[\overline{M}_{0,3}^{\kappa_1,\kappa_2,\kappa_3}(M^{\PGL_n},0)]^{\vir}=Eu\left(\bigoplus_{\substack{(a_1, a_2, a_3)\\
a_1+a_2+a_3=2}}L_{a_1}\oplus L_{a_2}\oplus L_{a_3}\right)\cap [\overline{M}_{0,3}^{\kappa_1,\kappa_2,\kappa_3}(M^{\PGL_n},0)],
\end{equation}
where $Eu$ stands for Euler class.   The  virtual fundamental class $[\overline{M}_{0,3}^{\kappa_1,\kappa_2,\kappa_3}(M^{\PGL_n},0)]^{\vir}$ has dimension 
$\dim(M^{\PGL_n})-\sum_{i=1}^{3}\age(\kappa_i)$ and the obstruction bundle has rank:
\begin{equation}\label{eqn_rank_obstruction-bundle}
\rk=\dim(\overline{M}_{0,3}^{\kappa_1,\kappa_2,\kappa_3}(M^{\PGL_n},0))-(\dim(M^{\PGL_n})-\sum_{i=1}^{3}\age(\kappa_i)).
\end{equation}

For our purpose, we describe in detail the calculation in \cite{BP} of the Chen-Ruan cup products for the case of $M^{\PGL_2}=[M/\Gamma]$. In this case $\Gamma\simeq (\zz/2\zz)^{2g}$. By Proposition \ref{prop_prime_n}, for a nontrivial $\kappa\in\Gamma$, the $\kappa$-fixed locus $M^{\kappa}$ is the smooth Prym variety $\Pic(X^\prime/X)$ of dimension $g-1$. It is a smooth abelian variety of dimension $g-1$.

Recall the presentation of $H^*_{I, \CR}([M/\Gamma])$ given in (\ref{eqn_Chen-Ruan_I}).
We describe the relations $I_{\orb}$ coming from Chen-Ruan cup products.

For any $0\neq \kappa\in \Gamma$ so that $[M^{\kappa}/\Gamma]=[\Prym(X^\prime/X)/\Gamma]$, and  for an even integer $s$ between $[0, 2(g-1)]$, let $h_s$ denote any class in $h_{\sum_{i=0}^{s-1}r^i_{\kappa}+1}, \cdots, h_{\sum_{i=0}^{s-2}r^i_{\kappa}+r_{\kappa}^s}$, i.e., the class in $H^{s}(\Prym(X^\prime/X))$. Then we have the following four cases. 
$$
\begin{cases}
(1)& \kappa_1=\kappa_2=\kappa\neq 0;\\
(2)& \kappa_1=0, \kappa_2=\kappa\neq 0;\\
(3)& \kappa_1=\kappa\neq 0, \kappa_2=0;\\
(4)& \kappa_1\neq \kappa_2\neq 0.
\end{cases}
$$
For a nontrivial $\kappa\in\Gamma$, $\age(\kappa)=(g-1)$ was calculated in \cite{BP}. In case (1), $\kappa_3=0$; and in cases (2), (3), $\kappa_3=\kappa$. In these three cases, we have $\overline{M}_{0,3}^{\kappa_1,\kappa_2,\kappa_3}(M^{\PGL_n},0))=[M^{\kappa}/\Gamma]$. It is easy to calculate 
the rank in (\ref{eqn_rank_obstruction-bundle}) to be $$\rk=(g-1)-(3g-3)-2(g-1)=0.$$
Therefore, the Chen-Ruan cup product is
\begin{equation}\label{eqn_Chen-Ruan_product_123}
\langle\alpha_1\cup_{\CR}\alpha_2, \alpha_3\rangle=
\int_{[M^{\kappa}/\Gamma]}\ev_1^*\alpha_1\cdot \ev_2^*\alpha_2\cdot \ev_3^*\alpha_3.
\end{equation}
We list the calculations here. 

(1). Since $\int_{[\Prym(X^\prime/X)/\Gamma]}h^{g-1}=\frac{1}{|\Gamma|}$,  we choose a cohomology class $\Omega\in H^*([M/\Gamma])$ such that $\int_{[M/\Gamma]}\Omega\cdot h_{(g-1)-\frac{1}{2}(s+t)}=\frac{1}{|\Gamma|}$.
We have 
$$
1_{\kappa}h_{s}\cup_{\CR}1_{\kappa}h_{t}=
\begin{cases}
\Omega, & s+t\le 2(g-1);\\
0 ,& \text{otherwise}.
\end{cases}
$$

(2).  Let $\iota: [M^{\kappa}/\Gamma]\hookrightarrow M^{\PGL_2}$ be the inclusion. 
We have $\alpha_1\cup_{\CR}\alpha_2=\iota^*\alpha_1\cdot \alpha_2$.  Thus, we calculate
$$
\alpha\cup_{\CR}1_{\kappa}h_{s}=
\begin{cases}
1_{\kappa}\cdot \alpha, & s< 2(g-1);\\
0 ,& s=2(g-1).
\end{cases}
$$
$$
\beta\cup_{\CR}1_{\kappa}h_{s}=
\begin{cases}
1_{\kappa}\cdot \beta, & s< 2(g-1)-2;\\
0 ,& s\ge 2(g-2)-2.
\end{cases}
$$
$$
\gamma\cup_{\CR}1_{\kappa}h_{s}=
\begin{cases}
1_{\kappa}\cdot \gamma, & s< 2(g-1)-4;\\
0 ,& s\ge 2(g-1)-4.
\end{cases}
$$

(3). Still let $\iota: [M^{\kappa}/\Gamma]\hookrightarrow M^{\PGL_2}$ be the inclusion. 
We have $\alpha_1\cup_{\CR}\alpha_2=\alpha_1\cdot \iota^*\alpha_2$.  The calculation is the same as (2) with the order reversed. 

(4).  In this case  we have $\overline{M}_{0,3}^{\kappa_1,\kappa_2,\kappa_3}(M^{\PGL_n},0))=[M^{\kappa_1}/\Gamma]\cap [M^{\kappa_2}/\Gamma]$. 
If $[M^{\kappa_1}/\Gamma]\cap [M^{\kappa_2}/\Gamma]$ is empty, then $\alpha_1\cup_{\CR}\alpha_2=0$. 

If $[M^{\kappa_1}/\Gamma]\cap [M^{\kappa_2}/\Gamma]$ is non-empty,  then from \cite{BP}, in the intersection product 
$$\mu: H^1(X, \zz/2\zz)\times  H^1(X, \zz/2\zz)\to H^2(X,\zz/2\zz)\cong\mu_2,$$
$\mu(\kappa_1, \kappa_2)\neq 0$ and the intersection $[M^{\kappa_1}/\Gamma]\cap [M^{\kappa_2}/\Gamma]$ is a stacky point $B(\mu_2\times\mu_2)$. It is easy to calculate the rank in (\ref{eqn_rank_obstruction-bundle}) to be $\rk=(3g-3)-3(g-1)=0$. Therefore, the Chen-Ruan cup product is
\begin{equation}\label{eqn_Chen-Ruan_product_123}
\langle\alpha_1\cup_{\CR}\alpha_2, \alpha_3\rangle=
\int_{B(\mu_2\times\mu_2)}\ev_1^*\alpha_1\cdot \ev_2^*\alpha_2\cdot \ev_3^*\alpha_3.
\end{equation}

Let $h_s^i (i=1,2,3)$ be the cohomology classes on $H^{s}([M^{\kappa_i}/\Gamma])$. 
We calculate 
$$\int_{[M^{\kappa_3}/\Gamma]}1_{\kappa_3}\cdot h^3_{2(g-1)}=\frac{1}{|\Gamma|}=\frac{1}{2^{2g}}.$$
Thus, we have 
$$
1_{\kappa_1}h^1_{s}\cup_{\CR}1_{\kappa_2}h^2_{t}=
\begin{cases}
2^{2g-2}\cdot 1_{\kappa_3}h^3_{2(g-1)}, & s=t=0;\\
0 ,& \text{otherwise}.
\end{cases}
$$
Therefore,  $I_{\orb}$ is the collection of relations coming from the Chen-Ruan products in (1), (2), (3) and (4). 


\section{Gromov-Witten invariants of moduli spaces of $\PGL_2$-bundles}\label{sec_GW}

In this section we calculate the genus zero Gromov-Witten invariants of $M^{\PGL_2}=[M/\Gamma]$, where 
$M=M_{2,L}(X)$ is the moduli space of stable rank $2$ bundles with fixed determinant 
$L\in \Pic^1(X)$, and $\Gamma=\Pic^0(X)[2]=(\zz/2\zz)^{2g}$ acts on $M$ by tensor product.  We fix $L$ to be degree $1$ since for any other odd degree $d$ $L\in\Pic^d(X)$, the moduli spaces are isomorphic. 

\subsection{Moduli spaces}\label{subsec_moduli_spaces}

\subsubsection{Stable map spaces}
Let $\overline{M}_{0,3}(M^{\PGL_2},e)$ be the moduli stack of stable maps $f: (\sC,\Sigma_i)\to M^{\PGL_2}$ such that 
\begin{enumerate}
\item $(\sC, \Sigma_i)$ is a twisted prestable genus zero curve with three marked gerbes $\Sigma_i\subset \sC$;
\item $f: \sC\to M^{\PGL_2}$ is stable map of degree $e=f^*(\hat{\Theta})$, where $\hat{\Theta}$ is the theta line bundle on $M^{\PGL_2}$;
\end{enumerate}
The stack $\overline{M}_{0,3}(M^{\PGL_2},e)$ has a decomposition into (unions of) connected components
$$\overline{M}_{0,3}(M^{\PGL_2},e)=\bigsqcup_{(\kappa_1, \kappa_2, \kappa_3)\in \Gamma^3}\overline{M}_{0,3}(M^{\PGL_2},e)^{(\kappa_1, \kappa_2, \kappa_3)},$$
where $\overline{M}_{0,3}(M^{\PGL_2},e)^{(\kappa_1, \kappa_2, \kappa_3)}$ is the locus of maps 
such that the evaluation maps $\ev_i: \overline{M}_{0,3}(M^{\PGL_2},e)\to IM^{\PGL_2}$ have images contained in $[M^{\kappa_i}/\Gamma]$.
The virtual dimension of $\overline{M}_{0,3}(M^{\PGL_2},e)^{(\kappa_1, \kappa_2, \kappa_3)}$ is
\begin{equation}\label{eqn_vd_stable_map}
\vd=(3g-3-3)+3+2e-\sum_{i=1}^{3}\age(\kappa_i).
\end{equation}

Let $\bbQ$ be the quantum parameter, which has degree $2$. Let $\alpha_i\in H^*([M^{\kappa_i}/\Gamma])$. The quantum orbifold product $\alpha_1\cup_{\bbQ}\alpha_2$ is defined by 
\begin{equation}\label{eqn_quantum_product_orbifold}
\langle\alpha_1\cup_{\bbQ}\alpha_2, \alpha_3\rangle=\sum_{e\ge 1}\langle \alpha_1, \alpha_2, \alpha_3\rangle_{0,3,e}^{M^{\PGL_2}}\bbQ^e
\end{equation}

\subsubsection{A projective bundle construction}\label{subsec_projective_bundle}

To better understand the moduli space of stable maps,  we introduce a projective bundle $N$ over the Jacobian\footnote{Note that, since $X$ is smooth and projective, the Jacobian of $X$ is compact.} $J=\Pic^0(X)$ of the curve $X$. Let $\{\phi_i\}$ be a symplectic basis of $H^1(J)$. Let $p: X\times J\to J$ be the standard projection and let $\tL\to X\times J$ be the universal line bundle. Then $c_1(\tL)=\sum (\xi_i\otimes \phi_i)\in H^1(X)\otimes H^1(J)$. Then 
$$c_1(\tL)^2=-2[X]\otimes \omega\in H^2(X)\otimes H^2(J),$$
where $\omega=\sum_{i=1}^{g}\phi_i\wedge \phi_{i+g}$ is the natural symplectic form for $J$.

The space $N$ is defined to be the projective bundle $\pp(\tE^{\vee})$, where 
$\tE=\tE xt^1_p(L\otimes \tL^{-1}, \tL)=R^1p_*(\tL^2\otimes L^{-1})$, where $L\in\Pic^1(X)$ is the fixed degree one line bundle on $X$. Thus, $N$ parametrizes nontrivial extensions 
$$0\to \sL\longrightarrow E\longrightarrow L\otimes \sL^{-1}\to 0,$$
where $\sL\in J$.  We can calculate that the groups 
$\Ext^1(L\otimes \sL^{-1}, \sL)=H^1(X, \sL^2\otimes L^{-1})=H^0(X, \sL^{-2}\otimes L\otimes K_X)$ are of constant rank equal to $g$. Thus, $N$ is a projective bundle over $J$ with fiber $\pp^{g-1}$. The bundle $E$ is stable, giving a morphism 
\begin{equation}\label{eqn_morphism_N_M}
\psi: N\to M=M_{2,L}(X).
\end{equation}

In general when $g\ge 3$, we have $\dim(M)=3g-3 > \dim(N)=2g-1$. When $g=2$,  $M$ and $N$ have the same dimension $3$. In this case we have:

\begin{prop}(\cite[\S 3]{Newstead})\label{prop_g=2_N}
When $g=2$, the morphism $\psi: N\to M$ is finite of degree $4$. 
\end{prop}
\begin{proof}
In \cite[\S 3]{Newstead},  for any point $\xi\in M$, there exists $4$ lines passing through $\xi$.  The number of lines in $M$ is isomorphic to the Jacobian $J$ of $X$.  Thus, just taking the natural projective $\pp^1$ bundle over $J$ we get the projective bundle $N$ and the $4:1$ map to $M$. 
\end{proof}

We give another explanation of the projective bundle $N$, generalizing the case of $g=2$ in \cite{Newstead}.  Let $V$ be the variety of the family of projective spaces $\pp^{g-1}$ on the moduli space $M=M_{2, L}(X)$.
Let $\widetilde{V}$ be the family itself, which means that $\widetilde{V}\subset V\times M$ is the subvariety consisting of points 
$\{(l, a))\on , l\cong \pp^{g-1}, a\in l\subset M\}$.  Let 
$$p_1: \widetilde{V}\to V,\quad  p_2: \widetilde{V}\to M$$
be the projections.  Let 
$p_2\times 1_{X}:  \widetilde{V}\times X\to M\times X$ be the projection map and 
$\sE_1=(p_2\times 1_{X})^*\sE$ be the pullback of the universal bundle. 
Since any rank $2$ stable bundle $E$ on $X$ is a nontrivial extension 
for  line bundles $\sL^\prime, \sL^{\prime\prime}$ such that $\deg(\sL^\prime)-\deg(\sL^{\prime\prime})=1$,  thus we have 
an extension
$$0\to N^\prime \rightarrow \sE_1\rightarrow N^{\prime\prime}\to 0$$
for the pullback universal vector bundle.   Now we restrict the extension to 
$p_1^{-1}(l)\times X$ and get an extension
\begin{equation}
0\to N^\prime_{l} \rightarrow \sE_{1,l}\rightarrow N^{\prime\prime}_{l}\to 0
\end{equation}
over $l\times X$.  Therefore, let $\pi_1: p_1^{-1}(l)\times X\to l$ and $\pi_2: p_1^{-1}(l)\times X\to X$, and we have  
$$N^\prime_{l}\cong \pi_1^*(H^r)\otimes \pi_2^*(\sL_l^\prime),$$
$$N^{\prime\prime}_{l}\cong \pi_1^*(H^s)\otimes \pi_2^*(\sL_l^{\prime\prime}),$$
where $H$ is the hyperplane bundle on $\pp^{g-1}$,  $r, s$ are integers and $\sL^\prime, \sL^{\prime\prime}$ are line bundles over $X$ which 
depend only on $l$.   Let $\sL^\prime$ has degree one and let $J_1$ be the variety of line bundles of degree $1$ over $X$. Then let 
$$\sigma: V\to J_1$$
be the map sending $l$ to $\sL_l^\prime$.  We have 
\begin{thm}\label{thm_map_sigma}
The map $\sigma: V\to J_1$ is an isomorphism. 
\end{thm}
\begin{proof}
From the morphism $\psi: N\longrightarrow M=M_{2,L}(X)$,  let $\overline{M}_{0,n}(N, 1)$ and $\overline{M}_{0,n}(M, 1)$ be the moduli spaces of stable maps from genus zero curves with $n$-marked points to the variety $N$ and $M$, then \cite[Corollary 6]{Munoz} proved that these two spaces are isomorphic. 
This means that the lines in $M$ can be understood as the lines in $N$.   The lines can only be in the fibers of the projective bundle $N$, which forces that the projective spaces $\pp^{g-1}$ lie in $M$.  From the dimensions the lines (inside the projective spaces $\pp^{g-1}$) in $M$ can not cover the whole $M$, but the variety of projective spaces $\pp^{g-1}$
must be equal to $J=J(X)$. 
\end{proof}
\begin{rmk}
It is interesting to have a direct proof of Theorem \ref{thm_map_sigma} by generalizing the proof in the case of $g=2$ in \cite{Newstead}. 
\end{rmk}

\subsection{Invariants}\label{subsec_invariants}

The finite abelian group $\Gamma$ also acts on $N$ by tensor product.  We have 
\begin{prop}\label{prop_N_M_equivariant}
There is a morphism $\psi: [N/\Gamma]\to [M/\Gamma]=M^{\PGL_2}$ between quotient stacks. 
\end{prop}
\begin{proof}
It suffices to check that the morphism $\psi: N\to M$ is $\Gamma$-equivariant, which is clear from the description of objects in $N$ as extensions.
\end{proof}

Let $\overline{M}_{0,3}([N/\Gamma],e)$ be the moduli stack of stable maps $f: (\sC,\Sigma_i)\to [N/\Gamma]$ such that 
\begin{enumerate}
\item $(\sC, \Sigma_i)$ is a twisted prestable genus zero curve with three marked gerbes $\Sigma_i\subset \sC$;
\item $f: \sC\to [N/\Gamma]$ is stable map of degree $e=f^*(K_N)$, where $K_N$ is the canonical line bundle of $N$;
\end{enumerate}
The stack $\overline{M}_{0,3}([N/\Gamma],e)$ has a decomposition into (unions of) connected components
$$\overline{M}_{0,3}([N/\Gamma],e)=\bigsqcup_{(\kappa_1, \kappa_2, \kappa_3)\in \Gamma^3}\overline{M}_{0,3}([N/\Gamma],e)^{(\kappa_1, \kappa_2, \kappa_3)}$$
where $\overline{M}_{0,3}([N/\Gamma],e)^{(\kappa_1, \kappa_2, \kappa_3)}$ is the locus of maps 
such that the evaluation maps $\ev_i: \overline{M}_{0,3}([N/\Gamma],e)\to I[N/\Gamma]$ have images contained in $[N^{\kappa_i}/\Gamma]$.
The virtual dimension of $\overline{M}_{0,3}([N/\Gamma],e)^{(\kappa_1, \kappa_2, \kappa_3)}$ is
\begin{equation}\label{eqn_vd_stable_map_N}
\vd=(2g-1-3)+3+ge-\sum_{i=1}^{3}\age(\kappa_i).
\end{equation}

We relate the moduli space of stable maps to the moduli space of stable bundles on surfaces.  
Let $\sC$ be a stacky $\pp^1$, i.e., a stacky curve $\sC$ with coarse moduli space $\pp^1$. 
We are mostly interested in the case that $\sC=\pp(1,2)$ and $\pp_{2,2}$.  Note that 
$\pp_{2,2}$ is not the weighted projective line $\pp(2,2)$ (which is $\mu_2$-gerbe over $\mathbb{P}^1$), but the orbifold $\pp^1$ with two stacky points $B\mu_2$.

Let $$\sS:=X\times\sC.$$ 
The surface $\sS$ has irregularity $q=g\ge 2$, geometric genus $p_g=0$, and canonical line  bundle $K_{\sS}=(\int_{\sC}K_{\sC})[X]+(2g-2)[\sC]$. Consider the line bundle $\hat{L}:=L\otimes \sO_{\pp^1}(1)$ on $\sS$. Then 
$$c_1(\hat{L})=[\pp^1]+[X].$$
The ample cone of $\sS=X\times \sC$ is 
$$\{a[\pp^1]+b[X]\on a, b >0\}.$$
Let $H=[X]+t[\pp^1]$ with $t<<0$ be a polarization close to $[X]$.  
We let $c_2=1$ and $$\MM:=\sM^{H}(c_1, c_2)$$ be the moduli space of $H$-stable rank $2$ bundles  on $\sS$ with topological invariants $c_1, c_2$. 
We have a similar result below.

\begin{prop}\label{prop_stable_bundle_M}(see \cite[Proposition 2]{Munoz})
The moduli space $\MM$ is a projective bundle $\MM=\pp(\tE_{\zeta}^{\vee})\to J$ over the Jacobian $J=J(X)$ of $X$, where $\tE_{\zeta}$ is a vector bundle on $J$ with usual Chern class 
$$\tilde{\Ch}(\tE_{\zeta})=N+8\omega.$$
The universal bundle $\scrV\to \sS\times \MM$ is given by 
$$0\to \sO_{\sC}(1)\otimes \tL\otimes \lambda\longrightarrow \scrV\longrightarrow L\otimes \tL^{-1}\to 0,$$
where $\lambda$ is the tautological line bundle for the projective bundle. 
\end{prop}
\begin{proof}
In the case that $\sC=\pp^1$, this is \cite[Proposition 2]{Munoz}. Mu\~noz \cite{Munoz} used wall crossing for the moduli space of stable bundles on the surface $\pp^1\times X$ and found that $\MM$ is the projective bundle.  For the polarization $H_0$ on $\sS=\sC\times X$ close to $\sC$, the moduli space $\sM^{H_0}(c_1, c_2)$ of $H_0$-stable bundles on $\sS$ is empty. This is a result in \cite{Qin}. There is only one wall given by $\zeta=-[\pp^1]+[X]$ (we understand here that $\zeta=2[X]-\hat{L}=[X]-c_1(\Lambda)$ as a divisor). 
Therefore, the moduli space $\MM=\sM^{H}(c_1, c_2)$ of $H$-stable bundles on $\sS$ is obtained by the wall crossing. Let $F$ be a divisor such that $2F-\hat{L}\equiv \zeta$, and $\scrF\to \sS\times J$ be the universal bundle parametrizing divisors homologically equivalent to $F$. Then $\scrF=\tL\otimes \sO_{\pp^1}(1)$. Let 
$$\pi: \sS\times J\to J$$
be the projection. Then \cite[Proposition 2]{Munoz} proved that 
$\MM=\pp(\tE_{\zeta}^{\vee})$, where 
$$\tE_{\zeta}=\tE xt^1(\sO(\hat{L}-F), \sO(\scrF)=R^1\pi_{*}(\sO(\zeta)\otimes \tL^2).$$
The moduli space $\MM$ is the set of vector bundles $E$ the extensions
$$0\to \sO_{\pp^1}(1)\otimes \sL\longrightarrow E\longrightarrow \Lambda\otimes \sL^{-1}\to 0$$
for $\sL\in J$.  The dimension of the fiber, calculated by the Riemann-Roch formula, is $2g-1$ for $\sC=\pp^1$. 

In the case of stacky $\pp^1$, the proof of the construction of the moduli spaces exactly generalizes. We only need to calculate the dimension of the fiber of the projective bundle using Riemann-Roch formula for the stack $\sS=\sC\times X$. 

\textbf{Case $\sC=\pp(1,2)$:}
Then $\sS=\pp(1,2)\times X$.  In this case the inertia stack 
$$I\sS=\sS\sqcup B\mu_2\times X,$$
where $B\mu_2$ corresponds to the single stacky point in $\pp(1,2)$. Let $c(\Lambda)=1+x$, and $c(\sO_{\sC}(1))=1+h$. 
We have
\begin{align*}
-N&=\chi(\Lambda^{-1}\otimes \sL^2\otimes \sO_{\sC}(1))\\
&=\int_{\sC\times X}\Ch(\Lambda^{-1}\otimes \sL^2\otimes \sO_{\sC}(1))\cdot \Td_{\sC\times X}\\
&+\int_{B\mu_2\times X}e^{2\pi i \frac{1}{2}}\Ch(\Lambda^{-1}\otimes \sL^2\otimes \sO_{\sC}(1))
\cdot \Td_{\sC\times X}|_{B\mu_2\times X}\\
&=\int_{\sC\times X}e^{-x}\cdot e^{h}\cdot (1+\frac{3h}{2}+\cdots)(1+(1-g)x+\cdots)\\
&+\int_{B\mu_2\times X}-(1-x+\cdots)(1+h+\cdots)\cdot \frac{1}{1+e^{-h}}\cdot \frac{2h}{1-e^{-2h}}\cdot (1+(1-g)x+\cdots)\\
&=\int_{\sC\times X}(\frac{3}{2}(1-g)hx-\frac{3}{2}hx+(1-g)hx-hx)-\int_{B\mu_2\times X}(\frac{1}{2}(1-g)x-\frac{1}{2}x)\\
&=\frac{5}{4}(1-g)-\frac{5}{4}-\frac{1}{4}(1-g)+\frac{1}{4}=-g.
\end{align*}

\textbf{Case $\sC=\pp_{2,2}$:}
Then $\sS=\pp_{2,2}\times X$.  In this case the inertia stack 
$$I\sS=\sS\sqcup B\mu_2\times X\sqcup B\mu_2\times X,$$
where the two  $B\mu_2$ correspond to the two stacky points in $\pp_{2,2}$.  
Still let $c(\Lambda)=1+x$, and $c(\sO_{\sC}(1))=1+h$. 
In this case we choose $\sO_{\sC}(2)$ in the extension. 
We have
\begin{align*}
-N&=\chi(\Lambda^{-1}\otimes \sL^2\otimes \sO_{\sC}(2))\\
&=\int_{\sC\times X}\Ch(\Lambda^{-1}\otimes \sL^2\otimes \sO_{\sC}(2))\cdot \Td_{\sC\times X}\\
&+\int_{B\mu_2\times X}e^{2\pi i \frac{1}{2}}\Ch(\Lambda^{-1}\otimes \sL^2\otimes \sO_{\sC}(2))
\cdot \Td_{\sC\times X}|_{B\mu_2\times X}\\
&+\int_{B\mu_2\times X}e^{2\pi i \frac{1}{2}}\Ch(\Lambda^{-1}\otimes \sL^2\otimes \sO_{\sC}(2))
\cdot \Td_{\sC\times X}|_{B\mu_2\times X}\\
&=\int_{\sC\times X}e^{-x}\cdot e^{h}\cdot (1+\frac{4h}{2}+\cdots)(1+(1-g)x+\cdots)\\
&+\int_{B\mu_2\times X}-(1-x+\cdots)(1+h+\cdots)\cdot \frac{2h}{1-e^{-2h}}\cdot \frac{2h}{1-e^{-2h}}\cdot (1+(1-g)x+\cdots)\\
&+\int_{B\mu_2\times X}-(1-x+\cdots)(1+h+\cdots)\cdot \frac{2h}{1-e^{-2h}}\cdot \frac{2h}{1-e^{-2h}}\cdot (1+(1-g)x+\cdots)\\
&=\int_{\sC\times X}(2(1-g)hx-2hx+2(1-g)hx-2hx)-\int_{B\mu_2\times X}((1-g)x-x)-\int_{B\mu_2\times X}((1-g)x-x)\\
&=(1-g)-1-(1-g)+1=0.
\end{align*}
Therefore, when $\sC=\pp(1,2)$, $N=g$ and $\MM$ is projective $\pp^{g-1}$-bundle.  When $\sC=\pp_{2,2}$,  $N=0$ and $\MM$ is isomorphic to $J$. 
\end{proof}

The group $\Gamma=\text{Pic}^0(X)[2]$ acts on $\MM$ by tensor product.  
Our goal is to relate the Gromov-Witten invariants of $M^{\PGL_2}$ to the Donaldson invariants of $\sS=\sC\times X$. 

We first have the following result on Gromov-Witten invariants which is similar to a result in \cite{Harder-Narasimhan} on the cohomology.
\begin{prop}\label{prop_invariants_MPGL-nontwisted-sector}
Let $\alpha_1, \alpha_2,\alpha_3\in H^*(M/\Gamma)$ be three cohomology classes in $[M/\Gamma]$. We still denote $\alpha_1, \alpha_2,\alpha_3\in H^*(M/\Gamma)$ as the cohomology classes in $H^*(M)$ under the pullbacks of the natural morphism $\pi: M\to M/\Gamma$. Then we have 
$$|\Gamma|\cdot\langle \alpha_1, \alpha_2,\alpha_3\rangle^{[M/\Gamma]}_{0,3,e}=\langle \alpha_1, \alpha_2,\alpha_3\rangle^{M}_{0,3,e}.$$
\end{prop}
\begin{proof}
We begin with constructing a morphism $$\rho: \overline{M}_{0,3}(M,e)\to \overline{M}_{0,3}([M/\Gamma],e).$$ In this Proposition we assume that marked points are non-stacky. Let $[f: C\to M]\in \overline{M}_{0,3}(M,e)$. Put $\tilde{C}:=C\times \Gamma$, where $\Gamma$ is considered as a discrete set. Define a moprhism $\tilde{f}: \tilde{C}\to M$ by $\tilde{f}|_{C\times \{\kappa\}}:=(\kappa\cdot)\circ f$. The $\Gamma$-action on $\tilde{C}$ given by $\kappa'\cdot (p,\kappa)=(p,\kappa'\kappa)$ makes the natural map $\tilde{C}\to C$ a trivial principal $\Gamma$-bundle. The morphism $\tilde{f}$ is $\Gamma$-equivariant. Taking stack quotient by $\Gamma$ yields a morphism $\hat{f}:C=[\tilde{C}/\Gamma]\to [M/\Gamma]$, which is an object of $\overline{M}_{0,3}([M/\Gamma],e)$. We define $\rho([f]):=[\hat{f}]$.

The group $\Gamma$ acts on $\overline{M}_{0,3}(M,e)$ by post composition. By construction, we have $\hat{((\kappa\cdot)\circ f)}=\hat{f}$ for any $\kappa\in \Gamma$. Therefore we obtain a morphism
\begin{equation}\label{eqn:mod_isom}
[\overline{M}_{0,3}(M,e)/\Gamma]\to \overline{M}_{0,3}([M/\Gamma],e).    
\end{equation}
We claim that (\ref{eqn:mod_isom}) is an isomorphism. Let $\hat{f}: \sC\to [M/\Gamma]$ be an object of $\overline{M}_{0,3}([M/\Gamma],e)$ and consider the cartesian diagram
\begin{equation}
\xymatrix{
\tilde{\sC}\ar[r]\ar[d] & M\ar[d]\\
\sC\ar[r]^{\hat{f}} & [M/\Gamma].
}    
\end{equation}
Since $\hat{f}$ is representable, the fiber product $\tilde{\sC}$ is a scheme. By assumption, $\sC$ is a tree all of whose irreducible components are genus $0$, and stacky points are only possible at nodes. An end $\sC_e\subset \sC$ of a branch of $\sC$ is thus an orbifold $\mathbb{P}^1$ with possibly one stacky point coming from the attaching node. Since such an orbifold $\mathbb{P}^1$ is simply connected \cite{BN}, it follows that the inverse image of $\sC_e$ in $\tilde{\sC}$ is the disjoint union of $|\Gamma|$ copies of $\sC_e$. Hence $\sC_e$ is a scheme, i.e. the attaching node is non-stacky. Removing $\sC_e$ and repeating this argument, we conclude that nodes of $\sC$ are all non-stacky, and $\sC$ is a scheme. Furthermore, $\tilde{\sC}\simeq \sC\times \Gamma$, and restricting the morphism $\tilde{\sC}\to M$ to any component $\sC\times \{\kappa\}$ yields an object of $\overline{M}_{0,3}(M,e)$. This yields the inverse of (\ref{eqn:mod_isom}) and shows that it is an isomorphism. 

Now we consider the following diagram
\begin{equation}
\xymatrix{
&\overline{M}_{0,3}(M,e)\ar[r]^{ev}\ar[d]\ar[dl]_{\rho} & M^{\times 3}\ar[d]\\
\overline{M}_{0,3}([M/\Gamma],e)&[\overline{M}_{0,3}(M,e)/\Gamma]\ar[r]^{\,\,\,\hat{ev}}\ar[l]^{\simeq} &[M/\Gamma]^{\times 3}.
}    
\end{equation}
The commutativity of the square follows from the above construction. Since the quotient map $M\to [M/\Gamma]$ pulls back the tangent bundle $T_{[M/\Gamma]}$ to the tangent bundle $T_M$, virtual pushforward implies 
\begin{equation}\label{eqn:vir_pushforward}
\rho_*[\overline{M}_{0,3}(M,e)]^{vir}=|\Gamma|[\overline{M}_{0,3}([M/\Gamma],e)]^{vir}    
\end{equation}
The Proposition follows.
\end{proof}


We now study $\overline{M}_{0,3}([M^{\PGL_2}],e)$ with stacky marked points. We only need to consider the case degree $e=1$ and one marked point is non-stacky.

\begin{lem}\label{lem_one_stacky_point}
In the moduli space $\overline{M}_{0,3}([M^{\PGL_2}],1)$, the case that the domain curve having just one stacky marked point does not exist. 
\end{lem}
\begin{proof}
Let $[f:\sC\to [M^{\PGL_2}]]\in \overline{M}_{0,3}([M^{\PGL_2}],1)$. If $\sC$ has one stacky marked point, then it is straightforward to see that $\sC$ has an irreducible component carrying one stacky point, which may come from stacky nodes. Let $\sC_0\subset \sC$ be such a component. The representable morphism $f|_{\sC_0}:\sC_0\to [M^{\PGL_2}]=[M/\Gamma]$ yields the cartesian diagram
\begin{equation*}
\xymatrix{
D_0\ar[r]\ar[d] & M\ar[d]\\
\sC_0\ar[r] &[M/\Gamma],
}    
\end{equation*}
where $D_0$ is a scheme, $D_0\to \sC_0$ is a principal $\Gamma$-bundle, and $D_0\to M$ is $\Gamma$-equivariant. The (disconnected) principal $\Gamma$-bundle $D_0\to \sC_0$ is classified by a group homomorphism 
\begin{equation}\label{eqn:rep_pi1}
\pi^{orb}_1(\sC_0)\to \Gamma.    
\end{equation}
If $\sC_0$ has only one stacky point, then $\sC_0\simeq \mathbb{P}(1,2)$ is simply-connected \cite{BN}. Hence $D_0\to \sC_0$ is the trivial $\Gamma$-bundle, i.e. $D_0\simeq \sC_0\times \Gamma$, which forces $\sC_0$ to be a scheme and cannot have a stacky point. This is a contradiction.
\end{proof}

It remains to consider the case of two stacky marked points. Let $\overline{M}_{0,2}(M^{\PGL_2},1)^{\kappa,\kappa^\prime}\subset \overline{M}_{0,2}([M^{\PGL_2}],1)$ be the locus of maps with two stacky marked points. Since we consider degree $1$ maps and there are only two marked points, stability implies that there are two loci,
\begin{equation*}
\overline{M}_{0,2}(M^{\PGL_2},1)^{\kappa,\kappa^\prime}=\overline{M}_{0,2}([M^{\PGL_2}],1)^m \coprod \overline{M}_{0,2}([M^{\PGL_2}],1)^n.    
\end{equation*}
Here $\overline{M}_{0,2}([M^{\PGL_2}],1)^m$ parametrizes maps with smooth domains and $\overline{M}_{0,2}([M^{\PGL_2}],1)^n$ parametrizes maps of the form $f: \sC_0\cup \sC_1\to [M/\Gamma]$ where $f$ is of degree $1$ on $\sC_0\simeq \mathbb{P}^1$, $\sC_1\simeq \mathbb{P}^1_{2,2}$ carries the two stacky marked points, $f$ is constant on $\sC_1$, and $\sC_0, \sC_1$ meet at a non-stacky node.

\begin{lem}\label{lem_two_stacky_point}
The moduli space $\overline{M}_{0,2}([M^{\PGL_2}],1)^m$ is smooth with expected dimension, and the moduli space $\overline{M}_{0,2}([M^{\PGL_2}],1)^n$ is also smooth and contributes zero to the Gromov-Witten invariants. 
\end{lem}
\begin{proof}
Let $[f:\sC\to [M^{\PGL_2}]]\in \overline{M}_{0,2}([M^{\PGL_2}],1)^m$. Then $\sC\simeq \mathbb{P}^1_{2,2}\simeq [\mathbb{P}^1/\mathbb{Z}_2]$ and we have the cartesian diagram
\begin{equation}\label{eqn:Gamma_bdle}
\xymatrix{
D\ar[d]\ar[r]^{\hat{f}} & M\ar[d]\\
\sC\ar[r]^{f} &[M/\Gamma],
}    
\end{equation}
In this case $\pi^{orb}_1(\sC)\simeq \mathbb{Z}_2$ and the $\Gamma$-bundle $D\to \sC$ corresponds to an injective group homomphism $\pi^{orb}_1(\sC)\to \Gamma$. The image of $\pi_1^{orb}(\sC)$ under this homomorphism is the subgroup $\langle\kappa\rangle\subset \Gamma$ generated by $\kappa$ (note that $\kappa\kappa'=1$). Put $\Gamma_0:=\Gamma/\langle \kappa\rangle$. Then (\ref{eqn:Gamma_bdle}) is composed of two cartesian diagrams:
\begin{equation}\label{eqn:Gamma_bdle2}
\xymatrix{
D\ar[d]\ar[r]^{\hat{f}} & M\ar[d]\\
\mathbb{P}^1\ar[r]\ar[d] & [M/\Gamma_0]\ar[d]\\
\sC\ar[r]^{f} &[M/\Gamma],
}        
\end{equation}
This implies that the scheme $D$ is a disjoint union $D=\coprod_{i=1}^{|\Gamma|/2} D_i$ and, for each connected component $D_i$, the map $D_i\to \sC$ is a $\mathbb{Z}_2$-cover and the composition $D_i\to\sC\to C$ with the coarse moduli space map $\sC\to C\simeq \mathbb{P}^1$ is ramified over the images of the two stacky marked points. By Riemann-Hurwitz, we have $D_i\simeq \mathbb{P}^1$. Since $f: \sC\to [M/\Gamma]$ is of minimal degree, the restriction of $\hat{f}: D\to M$ to $D_i$ defines a minimal rational curve in $M$. 

The arguments in Proposition \ref{prop_invariants_MPGL-nontwisted-sector} implies that 
\begin{equation*}
\overline{M}_{0,2}([M^{\PGL_2}],1)^m\simeq [\left(ev_0^{-1}(M^\kappa)\cap ev_\infty^{-1}(M^{\kappa'})\right)/\left(\text{PGL}_2(\mathbb{C})\times\Gamma_0\right)],    
\end{equation*}
where $ev_0, ev_\infty: \text{Hom}_1(\mathbb{P}^1,M)\to M$ are evaluation maps and $\text{PGL}_2(\mathbb{C})$ acts by reparametrizing the domain $\mathbb{P}^1$. By \cite[Theorem 3.1]{MS_2009}, $\text{Hom}_1(\mathbb{P}^1,M)$ is compact. Therefore $\overline{M}_{0,2}([M^{\PGL_2}],1)^m$ is compact. Its expected dimension is
$$(3g-3-3)+2+\frac{1}{2}\Theta\cdot c_1([M/\Gamma])-\age(\kappa_1)-\age(\kappa_2)=g-1.$$
Furthermore, we have 
\begin{equation}\label{eqn:vanishing_H1}
H^1(\sC, f^*T_{[M/\Gamma]})\simeq 
H^1(D, \hat{f}^*T_M)^\Gamma\simeq H^1(D_i, (\hat{f}|_{D_i})^*T_M)^{\mathbb{Z}_2}=0,    
\end{equation}
where the last equality uses the proof of \cite[Theorem 3.1]{MS_2009}. Hence $\overline{M}_{0,2}([M^{\PGL_2}],1)^m$ is smooth of expected dimension $g-1$.

A similar argument shows that 
\begin{equation*}
\overline{M}_{0,2}([M^{\PGL_2}],1)^n\simeq [\left(\text{Hom}_1(\mathbb{P}^1, M)\times_{ev_\infty, M^\kappa, id}M^\kappa\right)/\text{PGL}_2(\mathbb{C})]\simeq [ev_\infty^{-1}(M^\kappa)/\text{PGL}_2(\mathbb{C})]    
\end{equation*}
is smooth and compact. Its dimension is 
$$(3g-3-3)+\Theta\cdot c_1([M/\Gamma])-\text{codim}(M^\kappa\subset M)=3g-4-(2g-2)=g-2,$$
which is less than the expected dimension $g-1$ and does not contribute. 
\end{proof}

The forgetful morphism
\begin{equation*}
\mathsf{ft}:\overline{M}_{0,3}(M^{\PGL_2},1)^{0,\kappa,\kappa^\prime}\to \overline{M}_{0,2}(M^{\PGL_2},1)^{\kappa,\kappa^\prime}    
\end{equation*}
is the universal curve over $\overline{M}_{0,2}(M^{\PGL_2},1)^{\kappa,\kappa^\prime}$. Therefore
\begin{equation*}
\overline{M}_{0,3}(M^{\PGL_2},1)^{0,\kappa,\kappa^\prime}=\mathsf{ft}^{-1}(\overline{M}_{0,2}(M^{\PGL_2},1)^{m})\coprod\mathsf{ft}^{-1}(\overline{M}_{0,2}([M^{\PGL_2}],1)^n).    
\end{equation*}
Since
\begin{equation*}
\mathsf{ft}|_{\mathsf{ft}^{-1}(\overline{M}_{0,2}(M^{\PGL_2},1)^{m})}: \mathsf{ft}^{-1}(\overline{M}_{0,2}(M^{\PGL_2},1)^{m})\to \overline{M}_{0,2}(M^{\PGL_2},1)^{m}    
\end{equation*}
is a $\mathbb{P}_{2,2}$-fibration over the smooth and compact base $\overline{M}_{0,2}.(M^{\PGL_2},1)^{m}$, $\mathsf{ft}^{-1}(\overline{M}_{0,2}(M^{\PGL_2},1)^{m})$ is smooth and compact. Its dimension, which is equal to the expected dimension, is $g$.

Let us denote  by the moduli stack  
$\overline{M}_{0,3}([M^{\PGL_2}],1)^{m}:=\mathsf{ft}^{-1}(\overline{M}_{0,2}(M^{\PGL_2},1)^{m})$.

\begin{thm}\label{thm_invariants_MPGL}
Suppose that there are two nontrivial stacky points on the genus zero $3$-marked source twisted curve, then 
there is an isomorphism 
$$\Psi: \overline{M}_{0,3}([M^{\PGL_2}],1)^m\to \MM,$$
where $\MM$ is the moduli space in Proposition \ref{prop_stable_bundle_M}.
\end{thm}
\begin{proof}
Under the condition of the theorem  the source twisted curve $\sC=\pp^1_{2,2}$.
Let $f: \sC\to M^{\PGL_2}=[M/\Gamma]$ be a stable map from  $\sC$ to $[M/\Gamma]$ of degree $1$. Then from the definition of the moduli stack $M^{\PGL_2}=[M/\Gamma]$, we have the diagram 
(\ref{eqn:Gamma_bdle}).
From the former arguments, 
$\pi^{orb}_1(\sC)\simeq \mathbb{Z}_2$ and  (\ref{eqn:Gamma_bdle}) is composed of two cartesian diagrams
in (\ref{eqn:Gamma_bdle2}).

From Diagram (\ref{eqn:Gamma_bdle2}), $D$ is $D=\coprod_{i=1}^{|\Gamma|/2} D_i$ and $D_i\simeq \mathbb{P}^1$. The map $f$ corresponds to principle $\Gamma$-bundle $D\to \sC$ and a stable map 
$\hat{f}$.  Since $D$ is a disjoint union, this induces a family of vector bundles over $\pp^1\times X$ which is $\zz_2$-equivariant.  Thus,  
since $\pp^1\to \sC$ is a double cover, this induces a 
family $F$ of vector bundles over 
$\sC\times X$. Its first Chern class $c_1=c_1(\hat{L})=[\sC]+[X]$, and $c_2=1$. The bundle $F$ is stable over any point $p\in \sC$ since $2$ and $c_1\cdot H$ are coprime.

 We use $\overline{M}_{0,3}(M^{\PGL_2},1)^{0,\kappa_1,\kappa_2}=\overline{M}_{0,3}([M^{\PGL_2}],1)^m$ to represent this component of the moduli space of stable maps. We have 
$$\dim(\overline{M}_{0,3}(M^{\PGL_2},1)^{0,\kappa_1, \kappa_2})=(3g-3-3)+3+\frac{1}{2}\Theta\cdot c_1([M/\Gamma])-\age(\kappa_1)-\age(\kappa_2)=g.$$

From Proposition \ref{prop_stable_bundle_M}, the moduli space $\MM$ is smooth, and is a projective bundle over the Jacobian $J=J(X)$. It has dimensions $3g-1$, and $g$ in the cases $\sC=\pp^1$ and $\pp_{2,2}$ respectively. Thus, we obtain the same moduli spaces. 
\end{proof}

\subsection{Calculations-non twisted sector}\label{subsec_GW_invariants}

Recall that we have the following construction of the cohomology class in $\MM$ by considering the Chern classes of the universal vector bundle. We recall the notations in the introduction. Recall that $\{\xi_1, \cdots, \xi_{2g}\}\in H^1(X,\zz)$ is a basis and $\{\xi^*_1, \cdots, \xi^*_{2g}\}$ is the dual basis in $H_1(X,\zz)$.
Then $\xi^*_i\xi^*_{i+g}=[X]\in H_2(X,\zz)$ for $1\le i\le g$. We have $\psi=\sum_{i=1}^{2g}\xi_i\otimes \psi_i$. The classes $\alpha, \beta, \psi_i$ are given as follows. Let $$\mu: H_*(X)\to H^{4-*}(\MM)$$
be the map given by $\mu(a)=-\frac{1}{4} p_1(\mathfrak{g}_{\sE})/a$. Here $\mathfrak{g}_{\sE}\to X\times \MM$ is the associated universal $SO(3)$-bundle and $p_1(\mathfrak{g}_{\sE})\in H^4(X\times \MM)$ is the first Pontrjagin class. Then we have the classes 
$$
\begin{cases}
\alpha=2\mu([X])\in H^2(\MM), \\
\psi_i=\mu(\xi_i^*)\in H^3(\MM), & 1\le i\le 2g,\\
\beta=-4\mu(x)\in H^4(\MM),& x\in H_0(X) \text{~the point class.}
\end{cases}
$$
In the introduction we defined these classes for the moduli space $M$, but it is also holds for the moduli space of stable bundles $\MM$. Using the projective bundle $N\to J$ defined before, \cite{Munoz} proved that the Gromov-Witten invariants of $M$ is the same as the Donaldson invariants for $\sC\times X$. From Proposition \ref{prop_invariants_MPGL-nontwisted-sector}, we have 

\begin{prop}\label{prop_GW_M}(\cite[Theorem 11]{Munoz})
Let the genus $g=g(X)>2$ and $\alpha^{n_1}\beta^{n_2}\gamma^{n_3}$ has degree $3g-1$. Then we have 
$$\langle\alpha^{n_1}, \beta^{n_2}, \gamma^{n_3}\rangle_{0,3,1}^{[M/\Gamma]}=
(-1)^{g-1}D_{\sS, H, c_1}((2[X])^{n_1}(-4 x)^{n_2}(\sum \xi_{i}^*\cdot \xi_{i+g}^*)^{n_3}),$$
where $D_{\sS, H, c_1}((2[X])^{n_1}(-4 x)^{n_2}(\sum \xi_{i}^*\cdot \xi_{i+g}^*)^{n_3})$ is the Donaldson invariant for the surface $\sS=\sC\times X$.
\end{prop}

The Donaldson invariants $D_{\sS, H, c_1}((2[X])^{n_1}(-4 x)^{n_2}(\sum \xi_{i}^*\cdot \xi_{i+g}^*)^{n_3})$ is defined as follows. From the definition of the morphism $\mu: H_*(X)\to H^{4-*}(\MM)$, we have 
$$D_{\sS, H, c_1}((2[X])^{n_1}(-4 x)^{n_2}(\sum \xi_{i}^*\cdot \xi_{i+g}^*)^{n_3})=\epsilon_{\sS}(c_1)\int_{[\MM]}\alpha^{n_1}\cdot\beta^{n_2}\cdot\gamma^{n_3}.$$
Here the factor $\epsilon_{\sS}(c_1)=(-1)^{\frac{K_{\sS}c_1+c_1^2}{2}}=(-1)^{g-1}$ is from the calculation of Donaldson invariant.  

The invariant can further be reduced to the invariant 
$$(-1)^{g-1}\langle (4\omega+[X])^{n_1}([X]^{2})^{n_2}\gamma^{n_3}[X]^2, [J]\rangle$$ 
on the Jacobian $J=J(X)$ of $X$. Here $[X]^{2g-1+i}=\frac{(-8)^i}{i!}\omega^i$. 
Using this formula \cite{Munoz} found the quantum product for the quantum cohomology $QH^*(M)$ in Theorem \ref{thm_quan_coh_M}. 
Since the cohomology $H^*(M)\cong H^*(M/\Gamma)$ by the result of Harder-Narasimhan, and the dimension of the moduli space $\overline{M}_{0,3,1}([M/\Gamma])^{0,0,0}$ (where there is no stacky points in the domain curves) in Theorem \ref{thm_invariants_MPGL}, the invariants are the same as the Gromov-Witten invariants for $M$. 

\subsection{Calculations-twisted sector}\label{subsec_GW_invariants-twisted}

Now we do the calculation for the orbifold Gromov-Witten invariants involving the two cases $\sC=\pp(1,2)$, and $\sC=\pp_{2,2}$. 

We consider the class $\alpha\in H^2([M/\Gamma])$, and let $\kappa\in \Gamma$ be a nontrivial element so that $[M^{\kappa}/\Gamma]=[\Prym(X^\prime/X)/\Gamma]$. Let $0\le s\le 2(g-1)$ be an even integer. The cohomology $H^s(\Prym(X^\prime/X)^{\Gamma})$ of $[\Prym(X^\prime/X)/\Gamma]$ is generated by $\tiny\mat{c}2(g-1)\\s\rix$ generators.
We use $1_{\kappa}$ as the generator in $H^0(\Prym(X^\prime/X)^{\Gamma})$.
Let $r_{\kappa}^s:=\tiny\mat{c}2(g-1)\\s\rix$ and $R_{\kappa}:=\sum_{\substack{s=0\\
s \text{~even}}}^{2(g-1)}r_{\kappa}^s=2^{2g-3}$. We let 
$$1_{\kappa}, 1_{\kappa} h, \cdots, 1_{\kappa}h_{r_{\kappa}^2}, 1_{\kappa}h_{r_{\kappa}^2+1}, \cdots, 1_{\kappa}h_{r_{\kappa}^2+r_L^4}, \cdots, 1_{\kappa}h_{R_{\kappa}}$$
be the generators of the cohomology $H^{*}(\Prym(X^\prime/X)^{\Gamma})$.
We let $H^*_{I, \CR}([M/\Gamma])$ be the $\SP(2g,\zz)$ invariant Chen-Ruan cohomology. Then we have a presentation given in (\ref{eqn_Chen-Ruan_I})

\begin{equation*}
H^*_{I, \CR}([M/\Gamma])\cong \left(\qq[\alpha,\beta,\gamma]\oplus\bigoplus_{0\neq L\in\Gamma}\qq[1_{\kappa}, 1_{\kappa} h, \cdots, 1_{\kappa}h_{R_{\kappa}}]\right)/(I_g, I_{\orb})
\end{equation*}
where $I_{\orb}$ is the relations coming from the Chen-Ruan products in  \S \ref{subsec_Chen-Ruan}.

The cohomology group $H^s([M/\Gamma])=H^s(\Prym(X^\prime/X)/\Gamma)$ is generated by $r_{\kappa}^s:=\tiny\mat{c}2(g-1)\\s\rix$ generators. 

Let $h^s$ represent any class in $h_{(\sum_{i=0}^{s-2}r^i_{\kappa})+1}, \cdots, h_{(\sum_{i=0}^{s-2}r^i_{\kappa})+r_{\kappa}^s}$, i.e., any cohomology  class in $H^{s}(\Prym(X^\prime/X))$. 
We want to calculate $\alpha\cdot_{\bbQ}1_{\kappa} h^s$.
By definition, for any cohomology class $A\in H^*_{I, \CR}([M/\Gamma])$, we have 
$$(\alpha\cdot_{\bbQ}1_{\kappa} h^s, A)=\sum_{d\ge 0}\langle \alpha, 1_{\kappa} h^s, A\rangle_{0,3,d}^{[M/\Gamma]}\bbQ^{d}.$$

From dimension reason, only $d=0,1$ survive. For $d=0$, it is the classical Chen-Ruan product
$$\alpha\cup_{\CR}1_{\kappa} h^s=1_{\kappa}\alpha h^s\in H^{2+s}(M^{\kappa}/\Gamma).$$
Since the age is $\age(\kappa)=g-1$, this class $1_{\kappa}\alpha h^s\in H^{2+s}$ has orbifold degree $2+s+2(g-1)$. 

For $d=1$, we need to calculate $\langle \alpha, 1_{\kappa} h^s, A\rangle_{0,3,1}^{[M/\Gamma]}$ for all  $A\in H^*_{I, \CR}([M/\Gamma])$. 
First let  $A\in H^*([M/\Gamma])$, which is a class in the non-twisted sector.  Therefore, we need to consider the moduli space 
$\overline{M}_{0,3}([M/\Gamma], 1)^{0,\kappa,0}$, which has the (virtual) dimension 
$2g-1$ from Theorem \ref{thm_invariants_MPGL}.  
From Lemma \ref{lem_one_stacky_point}, this case contributes zero to the orbifold Gromov-Witten invariants. 

Second, let  $A\in H^*([M^{\kappa^\prime}/\Gamma])$, for $\kappa^\prime\in \Gamma$ another nontrivial element.  In this case we need to consider the moduli space $\overline{M}_{0,3}([M/\Gamma], 1)^{0,\kappa,\kappa^\prime}$, which has the (virtual) dimension $g$ from Theorem \ref{thm_invariants_MPGL}.  It is isomorphic to the Jacobian of $X$. This case only involves the stable maps from $\sC=\pp_{2,2}$.

We still have three evaluation maps:
$$
\begin{cases}
\ev_1: \MM\to [M/\Gamma];\\
\ev_2: \MM\to [M^{\kappa}/\Gamma]=[\Prym(X^\prime/X)/\Gamma];\\
\ev_3: \MM\to [M^{\kappa^\prime}/\Gamma]=[\Prym(X^\prime/X)/\Gamma].
\end{cases}
$$

In this case the evaluation maps $\ev_2: \MM\to [M^{\kappa}/\Gamma]=[\Prym(X^\prime/X)/\Gamma]$ and  $\ev_3: \MM\to [M^{\kappa^\prime}/\Gamma]=[\Prym(X^\prime/X)/\Gamma]$ are similar to the above case. The evaluation map $\ev_1: \MM\to [M/\Gamma]$ is actually $\ev_1: J(X)\to [M/\Gamma]$. From \cite{Munoz}, we have $\ev_1^*\alpha=\omega\in H^2(J)$ which is the symplectic form of $J$. 

We have the algebraic degrees
$$\deg(\alpha)=1;\quad  \deg(1_{\kappa}h^s)=(g-1)+\frac{1}{2}s,$$
so in order to make $\langle \alpha, 1_{\kappa} h^s, A\rangle_{0,3,1}^{[M/\Gamma]}$ non-zero,  $A$ must have degree 
$$\deg(A)=g-1-\frac{1}{2}s.$$
The pair $(\alpha\cdot_{\bbQ}1_{\kappa} h^s, A)$ is the integration on $[M^{\kappa^\prime}/\Gamma]$. 
So we take $A=1_{\kappa^\prime}\alpha^{g-1-\frac{1}{2}s}$, we have 
$$\int_{[M^{\kappa^\prime}/\Gamma]}\alpha^{\frac{1}{2}s}\cdot \alpha^{g-1-\frac{1}{2}s}=\constant.$$ 
Thus, we have 
\begin{equation}\label{eqn_key_calculation2}
(\alpha\cdot_{\bbQ}1_{\kappa} h^s, A)=(1_{\kappa^\prime}\alpha^{\frac{1}{2}s}, A)=\langle \alpha, 1_{\kappa} h^s, A\rangle_{0,3,1}^{[M/\Gamma]}
\end{equation}
up to a constant.

\subsection{The proof of Theorem \ref{thm_quantum_M-g>3}}

The quantum product of the class $\alpha$ with any cohomology class in 
$A\in H^*_{\CR}([M/\Gamma])$ is determined by the three-point invariant $\langle \alpha, A, B\rangle_{0,3,1}^{[M/\Gamma]}$ for 
$B\in H^*_{\CR}([M/\Gamma])$.
From Lemma \ref{lem_one_stacky_point}, the case that there is only one stacky marked point in the three marked points can not happen. 
So there are only case that these three marked points are non-stacky, and this is the case of non-stacky Gromov-Witten invariants in Proposition \ref{prop_invariants_MPGL-nontwisted-sector}. Thus, the relations (\ref{eqn_Quan_cohomology_relations-2-g>4}) and (\ref{eqn_Quan_cohomology_relations}) come from the quantum cohomology of $M$.  

The quantum orbifold products of $\alpha$ with the cohomology classes of twisted sectors are from the  key calculations in (\ref{eqn_key_calculation2}).  
Let $(\sC, p_1,p_2,p_3)\to [M/\Gamma]$ be a stable map of degree one. Then this is equivalent to giving a $\Gamma$-equivariant map 
$$\widetilde{\sC}\to M$$
where $\pi: \widetilde{\sC}\to \sC$ is a principal $\Gamma$-bundle over $\sC$. This covering map $\pi: \widetilde{\sC}\to \sC$ is determined by a morphism $\sigma: \pi_1(\sC)\to \Gamma$ from the orbifold fundamental group to $\Gamma$. In the case that there are two stacky points $p_2=B\mu_2, p_3=B\mu_2$, the two generators $\kappa, \kappa^\prime$ must satisfy the product to be $1$. Thus, $\kappa=\kappa^\prime$ in the calculation (\ref{eqn_key_calculation2}). 
Also the cohomology class $\alpha$ reduces to $h^2\in H^2([M^{\kappa}/\Gamma])$. 
$(h^2)^{\frac{1}{2}s}$ reduces one class $h^s$ in $H^s([M^{\kappa}/\Gamma])$.
Therefore, the quantum product calculations exactly match the relations in (\ref{eqn_Quan_cohomology_product_twisted-sector}).

\section{The case of genus $g=2$}\label{sec_234}

In this section we explicitly write down the quantum products for the case $M^{\PGL_2}=[M/\Gamma]$ when $g=2$. 

The moduli space $M=M_{2,L}(X)$ when $g(X)=2$ has dimension $3g-3=3$.  It is well-known that it is the intersection of two quadrics inside $\pp^5$, see \cite{Newstead}. 
Donaldson \cite{Donaldson} calculated the quantum cohomology of $M$ and we list his result.
The cohomology $H^*(M)$ is given by
$$H^*(M)=H^0(M)\oplus H^2(M)\oplus H^3(M)\oplus H^4(M)\oplus H^6(M)=\zz\oplus \zz\oplus \zz^2\oplus\zz\oplus \zz,$$
where $H^1(M)=H^5(M)=0$.  We use our notations before that 
$\alpha\in H^2(M)$, $\beta\in H^4(M)$, $\psi_1, \psi_2\in H^3(M)$, $\gamma=\psi_1\cdot\psi_2\in H^6(M)$ as generators. 
Note that in \cite{Donaldson},  Donaldson used the notations $h_2\in H^2(M)$, $h_4\in H^4(M)$, $\psi_1, \psi_2\in H^3(M)$, $h_6\in H^6(M)$ as generators.  We have 
$\alpha=h_2$, $\beta=-4h_4$ and $\gamma=4h_6$.
We have the following quantum products:
$$
\begin{cases}
\alpha\cdot_{\bbQ}\alpha=\beta+4\cdot 1_0 \bbQ;\\
\alpha\cdot_{\bbQ}\beta=\gamma+2\cdot  1_0 \alpha \bbQ;\\
\alpha\cdot_{\bbQ}\gamma=1_0 \beta \bbQ;\\
\alpha\cdot_{\bbQ}\psi_i=0.
\end{cases}
$$
Here we use $1_0$ to represent the identity class in $H^0(M)$. 
Using these quantum products it is not hard to see that the quantum cohomology $QH^*(M)$ of $M$ is generated by 
$\alpha$ and $\psi_1, \psi_2$, see \cite{Donaldson}, \cite[Example 24]{Munoz}.

For the moduli space $M^{\PGL_2}=[M/\Gamma]$ of stable $\PGL_2$-bundles,  we have $\Gamma=(\zz_2)^{4}$. 
So there are totally $7$ twisted sectors corresponding to the $7$ non-trivial elements in $\Gamma$. 
Let $\kappa\in\Gamma$ be a nontrivial element, the twisted sector $[M^{\kappa}/\Gamma]$ has dimension one. 
The Chen-Ruan cohomology $H^*_{I, \CR}(M^{\PGL_2})$ is generated by $\alpha,\beta,\gamma, 1_{\kappa}, 1_{\kappa}\alpha$ for all $0\neq \kappa\in \Gamma$.
Using the former results of the moduli spaces we calculate the quantum orbifold products
$$
\begin{cases}
\alpha\cdot_{\bbQ}\alpha=\beta+4\cdot 1_0 \bbQ;\\
\alpha\cdot_{\bbQ}\beta=\gamma+2\cdot  1_0 \alpha \bbQ;\\
\alpha\cdot_{\bbQ}\gamma=1_0 \beta \bbQ;\\
\alpha\cdot_{\bbQ}1_{\kappa}=1_{\kappa}\alpha+1_{\kappa}\bbQ;\\
\alpha\cdot_{\bbQ}1_{\kappa}\alpha=1_{\kappa}\alpha\bbQ.
\end{cases}
$$
Using the above quantum products it is not hard to see that the quantum orbifold cohomology $QH^*(M^{\PGL_2})$ is generated by $\alpha$ and $1_{\kappa}$ for all $0\neq \kappa\in\Gamma$.



\subsection*{}

\end{document}